\documentclass{amsart}

\usepackage{amscd,amssymb}
\usepackage{verbatim}
\usepackage{young}
\usepackage[vcentermath]{youngtab}
\usepackage{hyperref}
\usepackage{cite}

\newtheorem{theorem}{Theorem}[section]

\newtheorem{lemma}[theorem]{Lemma}
\newtheorem{proposition}[theorem]{Proposition}

\newtheorem{conjecture}[theorem]{Conjecture}

\theoremstyle{definition}
\newtheorem{definition}[theorem]{Definition}

\newtheorem{remark}[theorem]{Remark}

\renewcommand{\qed}{\hfill \mbox{\raggedright \rule{0.1in}{0.1in}}}

\newcommand{\I}{i_2}
\newcommand{\blah}{i_3}
\newcommand{\III}{i_4}

\newcommand{\IIIIII}{i_r}
\newcommand{\IIIIIII}{i_1}

\newcommand{\cc}{\cdot}
\newcommand{\sse}{\subset}

\newcommand{\mc}{\mathcal}

\newcommand{\bb}{\mathbb}
\newcommand{\al}{\alpha}

\newcommand{\II}{\mathrm{I}}

\address{Massachusetts Institute of Technology, 428 Memorial Drive, Cambridge MA}
\email{rumenrd@mit.edu}

\begin{document}

\title[Lower Central Series]
{On the Maximal Containments of Lower Central Series Ideals}
\author[Rumen Dangovski]{Rumen Rumenov Dangovski}

\thanks
{The research of the author began at the Research Science Institute, Summer 2013, and was a part of the Undergraduate Research
Opportunities Program at the Massachusetts Institute of Technology, Fall 2014 and Spring 2015}

\begin{abstract}
We begin a discussion about the maximal containments of lower central series
ideals: ideals generated by products of two--sided ideals of the lower central
series of the free associative algebra on $n$ generators. We introduce two new
ideas to the topic, the PBW grading and the pigeonhole principle, that help us give a complete classification of the containments for $n=2$ and obtain
partial results in the general case.
\end{abstract}

\maketitle

\section{Introduction}

\subsection{Preliminaries}

In this paper all vector spaces are taken over a field $\mathcal{K}$ of
characteristic zero. Let $A_n$ (simply $A$ if $n$ is not specified) be the free
associative unitary algebra on $n$ generators $x_1,\dots,x_n.$ An \emph{element}
means a polynomial in $A_n$ if not specified. The bracket operation
$[a,b]=ab-ba$ equips $A_n$ with the structure of a Lie algebra. We norm commutators
to the right, i.e. $$[a_1,a_2,\dots, a_m]\equiv [a_1,[a_2,\dots,a_m]].$$ 

A \emph{pure commutator} is a commutator with single generators in its slots. For
example, $[x_1,x_2,x_3]$ is a pure commutator, while
$[x_1+7x_2,3x_1^2x_3,x_4,2x_5x_6]$ is not. The \emph{length} of a commutator is
its number of slots. In the above example the latter commutator has length four.
We also set the following convention: a generator $x_i$ in $A_n$ is itself a pure commutator of length one. A \emph{pure--commutator product} is a finite product of pure commutators. For example, $x_1x_2x_3^2[x_1,x_5,x_7][x_1,x_2][x_1,x_7]$ is a pure-commutator product consisting of 7 commutators ($x_1,$ $x_2,$ two pure commutators $x_3$, $[x_1,x_5,x_7],$ $[x_1,x_2]$ and $[x_1,x_7]$), 
while $x_1[x_1,x_2]+x_6[x_2,x_3][x_3,x_4]$ and $x_1[x_1,x_2][x_1,x_2][x_1,x_2^2]$ are not such products. We do not count $0$ as a pure commutator, nor a pure--commutator product, i.e. the commutators in the product cannot consist only of one generator, unless they are of length one.

Unless it is specified we do not care about the specific generators in the commutators in pure products, and often we will only introduce a sequence $(i_1,\dots,i_k),$ where $k$ is the number of commutators in the product and $i_j$ is the length of a pure commutator for $j=1,\dots,k.$ The above--mentioned example of pure--commutator product corresponds to the 7-tuple $(1,1,1,1,3,2,2).$

\subsection{Motivation}

The \emph{lower central series} $\{L_i(A_n)\}_{i\geq 1}$ of $A_n$ is defined by
$$L_1(A_n):=A_n \text{ and } L_{i}(A_n):=[A,L_{i-1}(A_n)]$$ 
for $i\geq 2.$ Unless $n$ is
specified $L_i\equiv L_i(A_n),$ and essentially $L_i$ is spanned by all
commutators of length $i$ with entries in $A_n.$ The series
$\{M_{i}(A_n)\}_{i\geq 1}$ is the series of the two--sided (one sided due to the
Leibniz rule) ideals $M_i=A\cdot L_i \cdot A$ generated by $L_i.$ Note that we
can abbreviate $M_i\equiv M_i(A_n).$

The lower central series ideals form an important filtration $\mc{M}$ in non--commutative
algebra
$$
A = M_1 \supset M_2 \supset M_3 \cdots
$$
It helps us ``approximate'' the algebra $A$ through the quotients $N_i:=M_i/M_{i+1},$ i.e.
$$
A=\mathrm{Gr}_{\mc{M}}(A)=\bigoplus_{i\geq 1}^{\infty}M_i/M_{i+1}=\bigoplus_{i\geq 1}^{\infty}N_i
$$

Therefore, in order to better understand this commutator-graded structure of $A,$ we need to study the $N$--ideals, i.e. to be aware of the structure of the $M$--ideals. This gives rise to the consideration of containment of products of $M$--elements.  
By \emph{lower central series ideals} in this paper we mean finite products of
$M$--elements $M_{i_1}\cdots M_{i_k},$ where $i_p$ is greater or equal to two for
$p=1,\dots, k.$ There are a few results on inclusions of products of
$M$--elements:
\begin{theorem}[Jennings, 1947, \cite{J}] \label{theorem7}
For every $n$ and $k$ there exists $m,$ such that $$\left(M_2(A_n)\right)^m
\subset M_k(A_n).$$
\end{theorem}
\begin{theorem}[Latyshev, 1965, \cite{Lat}, and Gupta and Levin, 1983, \cite{GL}] \label{theorem3}
For all $m$ and $l$ $$M_m \cdot M_l \subset M_{m+l-2}.$$
\end{theorem}
\begin{theorem}[Bapat and Jordan, 2010, \cite{BJ}] \label{theorem2}
For odd $m$ or odd $l$ $$M_m \cdot M_l \subset M_{m+l-1}.$$
\end{theorem}
\begin{theorem}[Bapat and Jordan, 2010, \cite{BJ}] \label{theorem6}
For odd $l$
$$
[M_l,L_k]\subset L_{k+l}.
$$
\end{theorem}

\begin{remark}
In ~\cite{kdrumen} (Section 1. Introduction) one may find a brief historical discussion of these results. In particular, we thank the authors for pointing out that it was Latyshev to first prove Theorem \ref{theorem3}.
\end{remark}

Essentially, these results allow us not to be redundant, i.e. to omit a few
inductions that can be found in the paper of Bapat and Jordan \cite{BJ}.
None of the known results ask whether the inclusions are maximal. This gives
rise to the discussion that follows.

We say that a pure-commutator product $p$ has \emph{a maximal inclusion in} $M_l(A_n)$ if $p$ is in $M_l(A_n)$ but $p$ is not in $M_{l+1}(A_n).$ What follows is the main definition of the paper.

\begin{definition}
For an associative unitary algebra $R$ and a $k$-tuple $i=(i_1,\dots,i_k),$ such
that
$i_p\geq 2$ for $p=1,\dots,k$ let $$
\II (R,i) \in \mathbb{N}
$$
be the maximal index such that
$$
M_{i_1}(R)\cdots M_{i_k}(R)\subset M_{\II (R,i)}(R),
$$
while
the product $M_{i_1}(R)\cdots M_{i_k}(R)$ is not contained in $M_{\II
(R,i)+1}(R).$
\end{definition}

For example, later in this paper we show that $\II(A_2,(2,2))=3,$ i.e.
$M_2(A_2)\cdot M_2(A_2) \subset M_3(A_2),$ but $M_2(A_2)\cdot M_2(A_2)
\nsubseteq M_4(A_2).$ Note that in this example $R=A_2$ and $i=(2,2)$ is a
2-tuple. The only results from the above-mentioned that we can apply is Theorem
\ref{theorem3}, which only tells us the trivial $$M_2(A_2)\cdot M_2(A_2) \subset
M_2(A_2).$$

\subsection{Main results}

The results presented in the Motivation imply the following bounds: 
\begin{itemize}
\item Theorem \ref{theorem3} implies $\II(A_n,(m,l))\geq m+l-2;$
\item Theorem \ref{theorem2} implies $\II(A_n,(m,l))\geq m+l-1$ when $m$ or $l$ is odd.
\end{itemize}

To our knowledge, prior to this paper there are no attempts to obtain an upper bound on $\II(A_n,i).$ In this paper we find an upper bound, which tells us that $\II(A_n,(m,l))\geq m+l-1$ is sharp for some partial cases. More specifically, we provide a complete classification of the lower central series containments for $A_2.$ 

\begin{theorem}[Complete classification for $n=2$] \label{theorem1}
\label{theorem1}
Given the $k$-tuple
$i=(i_1,\dots,i_k),$ such that $i_p$ is greater than or equal to two for
$p=1,\dots ,k$ the following holds
$$
\II(A_2,i)=i_1+\dots+i_k-k+1.
$$
\end{theorem}

This is our main theorem. Finding $\II(R,i)$ would characterize the
\emph{infrastructure} of
the lower central series ideals in $R$ in more detail. First of all, in Section \ref{grading}
through the PBW theorem we describe an unconventional grading of $A_n$ that
helps us prove non-inclusions. Second of all,
in Section \ref{inclusions} we obtain Theorem \ref{theorem1} via the simple pigeonhole principle. 
Moreover, we prove that the ideals $M_r(A_2)$ are finitely generated as two-sided ideals. We give an explicit set of generators. 
We conclude the section with a result on universal Lie nilpotent associative algebras. 
In Section \ref{discussion} we discuss the case of $n=3.$ In Section \ref{approach} we present an idea of non-containment, as well as a conjecture for maximal inclusions for the case of the tuple $(2,\dots,2) \in \bb{N}^{k}.$ In Section \ref{bounds} we combine the theorems, known previously, and our results. In Section \ref{applications} briefly we show applications to PI theory of the results in this paper.

\section{The PBW grading. Non-containment of ideals.} \label{grading}

Our proofs of non-containment rely on a nontrivial PBW grading of $A_n.$

Let $\mathfrak{g}_n$ be the free unitary Lie algebra on $n$ generators
$x_1,\dots x_n$ with an ordering $\leq$ of the basis
$B=\{1,x_1,\dots,x_n,[x_1,x_2],[x_1,[x_1,x_2]],\dots\},$ in which the pure commutators correspond to Lyndon words. Note that the bracket
$[,]$ of $\mathfrak{g}_n$ is not necessarily defined by $[a,b]=ab-ba.$ According
to the PBW Theorem, the set $
\{ b_{i_1}\cdots b_{i_m} \mid m\geq 0, b_{i_j}\in B, j \leq m, i_1\leq \cdots
\leq i_m \}
$
forms a basis of the universal enveloping algebra $\mathcal{U}(\mathfrak{g}_n).$
The PBW filtration
$$
\mathcal{F}_0 \subset \mathcal{F}_1 \subset \mathcal{F}_2 \cdots
$$
is defined by $\mathcal{F}_0=\mathcal{K},$ and $\mathcal{F}_r=\{a_{1}\cdots
a_{k} \mid a_{i} \in \mathfrak{g_n}, k \leq r\}$ for $r$ greater than or equal
to one.
Hence, from the fact that both of the algebras (as vector spaces) satisfy the same universal property 
$\mathrm{Gr}_{\mathcal{F}}(\mathcal{U}(\mathfrak{g}_n)) \cong
\mathcal{K}[x_1,\dots,x_n].$
There is a natural isomorphism between $A_n$ and $\mathcal{U}(\mathfrak{g}_n)$
defined by mapping generators. The isomorphism equips $A_n$ with PBW grading as follows:
$\mathcal{K}$ sits in degree zero; $\mathfrak{g}_n$ sits in PBW degree one; a
product of $r$ elements in $\mathfrak{g}_n$ sits in PBW degree $r$ for $r$
greater than or equal to two.

For example, for the ordering $\leq$ the monomial basis of $A_2$ up to (normal,
monomial) degree three is
\begin{center}
\begin{tabular}{c | l}
Monomial degree & Elements \\
\hline
Zero & 1 \\
One & $x_1, x_2$ \\
Two & $x_1^2, x_2^2, x_1x_2, x_2x_1$ \\
Three & $x_1^3, x_2^3, x_1^2x_2^1, x_1x_2^2, x_2^1x_1^2, x_2^2x_1 $
\end{tabular}
\end{center}
However, the PBW basis of $A_2$ up to PBW degree three is \begin{center}
\begin{tabular}{c | l}
PBW degree & Elements \\
\hline
Zero & 1 \\
One & $x_1,x_2,[x_1,x_2],[x_1,[x_1,x_2]],[x_1,[x_2,[x_1,x_2]]],\dots$ \\
Two &
$x_1^2,x_2^2,x_1x_2,x_1[x_1,x_2],[x_1,x_2][x_1,[x_2,x_1]],x_2[x_1,[x_2,[x_1,x_2]]],\dots$
\\
Three & $x_1^3, x_2^3, x_1x_2x_1, [x_1,x_2][x_1,x_2]x_2,
[x_1,[x_2,x_1],[x_1,x_2]]x_1[x_1,x_2],\dots$
\end{tabular}
\end{center} 
The commutative ordering for the normal degree is not necessary, while it is for
the PBW degree (due to the PBW theorem). The simplest case is
$x_2x_1=x_1x_2-[x_1,x_2]$ where the error term $[x_1,x_2]$ is of lower PBW
degree one.

\begin{lemma}[The PBW argument] \label{lemma1}
For $r\geq 2$ and $k\geq 0$ any element $m$ in $L_r(A_n)$ of standard degree
$r+k$ has a PBW degree equal to $k+1.$
\end{lemma}

\begin{proof}
Due to bi-linearity it suffices to consider only a single commutator $m$ with monomials in its slots (we can also ignore the scalars in $\mathcal{K}$). 
It is obvious that $m$ is a pure commutator only for $k=0,$ because it has
$r$ slots and degree $r+k.$ Hence when $k=0,$ $m\in \mathfrak{g}_n$ and thus $m$
has a PBW degree one, which is $k+1$ ($k=0$).

When $k\geq 1,$ w.l.o.g. in some of the slots of $m$ there is a product of $v$
generators $x_{i_1}\cdots x_{i_v}$, $v\geq 2.$ Set $a:= x_{i_1}\cdots x_{i_{v-1}}$ and
$b:=x_{i_v}.$ Consider a product $ab$ of elements $a$ and $b$ in a slot of $m:$
$$
m=[a_1,[\dots,[a_{r'},[ab,\ell]]]\dots]=\mathrm{ad}(a_1)(\cdots
\mathrm{ad}(a_{r'})(\mathrm{ad}(ab)(\ell))\cdots),
$$
where $\ell$ is a commutator of length $r-r'-1.$ From the Leibniz rule
$\mathrm{ad}(ab)(\ell)=a\mathrm{ad}(b)(\ell)+\mathrm{ad}(a)(\ell)b$, and
$$
m=\mathrm{ad}(a_1)(\cdots
\mathrm{ad}(a_{r'})(a\mathrm{ad}(b)(\ell))\cdots)+\mathrm{ad}(a_1)(\cdots
\mathrm{ad}(a_{r'})(\mathrm{ad}(a)(\ell)b)\cdots)
$$
Set $b':=\mathrm{ad}(b)(\ell)$ and $c':=\mathrm{ad}(a)(\ell).$ We reduce the number of adjoining operations by one and keep a product of
two elements, namely $ab'$ and $c'b).$ By
induction on $r$ a product of two
elements in a slot of $m$ expands to a linear combination of commutators with
products of less (normal) degree. This leads us to a product of $k+1$ pure
commutators in $\mathfrak{g_n},$ i.e. a product of $k+1$ elements of PBW degree one. From here the
statement of the lemma follows.
\end{proof}
For instance, when $r=3$ and $k=1,$ $[x_1,x_2,x_3x_4]$ expands to
$$[x_1,x_3][x_2,x_4]+x_3[x_1,x_2,x_4]+[x_2,x_3][x_1,x_4]+[x_1,x_2,x_3]x_4,$$
which has PBW degree two. We use Lemma \ref{lemma1} to prove non-containment of
products of lower central series ideals. We know that $M_l(A_n)\cdot M_m(A_n)
\subset M_{n+m-1-\epsilon}(A_n)$ where $\epsilon$ is zero or one, depending on
the parity of $l$ and $m.$ It is natural to conjecture that $M_l \cdot M_m$ is
not in $M_{l+m}.$ The idea is to suppose this is true, and take the lowest
degree part of the LHS and the RHS. The lowest degree part of a lower central
series ideal is a pure commutator. Thus the LHS has a product of two such
commutators that has PBW degree two while the RHS has a single pure commutator
of PBW degree one. This is a contradiction. We extend this idea to the following
result.
	
\begin{proposition} \label{proposition1}
Given $n$ greater than or equal to two and a $k$-tuple $i=(i_1,\dots,i_k)$ in $\bb{N}^k$ the
following holds
$$
\II(A_n,i)\leq 1-k+\sum_{p=1}^{k}i_p.
$$
\end{proposition}

\begin{proof}
It suffices to show that
$$
L_{i_1}\cdots L_{i_k} \nsubseteq M_{i_1+\cdots+i_k-k+2}.
$$
The lowest degree part of the LHS has the form $\ell_1\cdots \ell_k,$ where
$\ell_j \in \mathfrak{g}_n$ is a pure commutator in $L_j$ for $j=1,\dots,k.$
Therefore, the LHS has a PBW degree $k.$

The lowest degree part of the RHS, w.l.o.g., has the form $a \cdot \ell,$ where $\ell \in
L_{i_1+\cdots+i_k-k+2}.$ The normal degree is $\sum_{p=1}^{k}i_p,$ due to the
LHS. Hence, we may think we start from a pure commutator and insert $k-2$
generators in $a$ or in the slots of the commutator. If the normal degree of $a$
is $p$ ($p=0,1,\dots,k-2$) the normal degree of $\ell$ is
$\sum_{j=1}^{k}i_j+k-2-p.$ Now, we apply Lemma \ref{lemma1} to $\ell$ to prove
that the RHS breaks into components of PBW degree less than or equal to $k-1.$
This is incompatible with the PBW degree of LHS.
\end{proof}

\begin{remark}
Note that this is the first general result to give us an upper bound on the function $\II.$
\end{remark}

\section{Maximal containment for $n=2.$ Generators of $M$--ideals } \label{inclusions}

Our proofs of containment rely on the following three identities true for all free associative unitary algebras: 
\begin{itemize}
\item $[ab,c]=a[b,c]+[a,c]b$ --- The Leibniz Law;
\item $[a,b,c]+[c,a,b]+[b,c,a]=0$ --- The Jacobi Identity;
\item $[a,b][a,c]=a[a,b,c]+[a,b,a]c+[a,ac,b]$ --- The Pigeonhole Identity \footnote{This identity can be derived
from expansion of $[a,ab,c]$ via the Leibniz rule.}.
\end{itemize}

\subsection{Proof of Theorem \ref{theorem1}}

The PBW argument gives us the upper bound. It suffices to show that it is sharp, i.e. to prove the following result.

\begin{theorem}[Containment for $n=2$] \label{max2var}
For $i$ and $j$ greater than or equal to two 
$$M_i(A_2) \cdot M_j(A_2) \subset M_{i+j-1}(A_2).$$
\end{theorem}

In this section we set $M_i \equiv M_i(A_2)$ for any $i$ and $A \equiv A_2.$ We
start with a lemma that relies on the pigeonhole principle.

\begin{lemma} \label{lem1}
$[M_k,L_1] \sse M_{k+1}$ for $k\geq 2.$
\end{lemma}

\begin{proof}
When $k$ is odd the proof follows from Bapat and Jordan's result (in this case we even have the stronger condition $[M_k,L_1] \sse L_{k+1}).$ Hence, it suffices to consider $k$ even. Take a general generator $m=[a[b,l],c] \in [M_k,L_1]$ where $a,b,c \in A$ and $l
\in L_{k-1}.$ Then
$$
m=[b,l][a,c]+a[[b,l],c]=[b,l][a,c] \mod M_{k+1}.
$$
If $a$ is in the field $\mathcal{K}$ then $[a,c]=0,$ i.e. $m = 0 \mod M_{k+1}.$
Otherwise, due to bi-linearity let $a$ and $c$ be monomials. Similarly, w.l.o.g.
$b$ is a monomial. If $a$ and $c$ are monomials in only one generator then $m =
0 \mod M_{k+1}.$ Otherwise their expressions as products of generators differ by
at least one generator. Hence, since $n=2$ $b$ shares a generator with either
$a$ or $c.$ Due to anti-commutativity the common generator is between $b$ and $a$ and w.l.o.g. suppose it is $x_1.$ 


Further suppose $b=b'x_1b'',$ and $a=a'x_1a''.$ Using the Leibniz rule we can
reduce the degree of $b$ and $a$ by removing the \emph{extra} parts
$b',b'',c',c''.$ Induction on the degree of the commutators provides us with the
following transformation
\begin{equation}\label{B}
[b,l][a,c] \longrightarrow
[x_1,l]d[x_1,c],
\end{equation}
where due to the multilinearity, w.l.o.g. we set $d\in A_2$ to be a monomial. Using the Leibniz we obtain:
\begin{equation} \label{A}
[x_1,l]d[x_1,c] \longrightarrow  l[x_1,d][x_1,c]+[x_1,l
d][x_1,c] 
\end{equation}
We apply the pigeonhole identity to the RHS of Equation $\eqref{A}$ to
get products of the following kinds (we ignore scalars, signs, and due to Leibniz we can multiply by $L_1$ only to the left side):
\begin{itemize}
\item $l(x_1[x_1,d,c]+[x_1,d,x_1]c+[x_1,x_1c,d]) \in L_{k-1} \cdot M_3 \sse M_{k+1};$
\item $x_1[x_1,ld,c]+[x_1,ld,x_1]c+[x_1,x_1c,ld] \in L_1 \cdot [L_1,M_{k-1},L_1] \sse L_1 \cdot [L_1,L_{k}] \sse M_{k+1},$
\end{itemize}
because $k-1$ is odd and Bapat and Jordan's $[M_{k-1},L_1]\sse L_{k}$ holds.
\end{proof}

Note that, due to the pigeonhole principle, a simple case of this lemma, $[M_2,L_1]\sse M_3$ is true for all algebras that admit $M_2\cdot M_2 \sse M_3.$ We conjecture a stronger version of the lemma (although not particularly needed in our proof).

\begin{conjecture}
$[M_k,L_1]=L_{k+1}$
\end{conjecture}

\begin{theorem} \label{th1}
For any $k$ and $j$ $[M_k,L_j] \sse M_{k+j}.$
\end{theorem}

\begin{proof}
We use induction on $j.$ The base case $j=1$ follows from Lemma \ref{lem1}.
Suppose the $[M_k,L_{j'}] \sse M_{k+j'}$ for any $j' \leq j$ and $k$ arbitrary.
Consider a general generator $m=[m',[a,l]]$ of $[M_k,L_{j+1}],$ where $m' \in
M_{k}$ $a \in L_1,$ and $l \in L_{j}.$ By the Jacobi identity and
anti-commutativity
$
m=[[m',a],l]+[[l,m],a].
$
From Lemma \ref{lem1} $[m',a] = 0 \mod M_{k+1},$ hence by assumption $[[m',a],l]
= 0 \mod [M_{k+1},L_j] = 0 \mod M_{k+j+1}.$ Similarly $[l,m] = 0 \mod M_{k+j},$
i.e. $[[l,m],a] = 0 \mod M_{k+j+1}.$ Hence, $m = 0 \mod M_{k+j+1}.$

\end{proof}

To complete the proof of Theorem \ref{theorem1} we have that $M_i \cdot M_j \sse M_{i+j-1}$ when $i$ or $j$
is odd (Bapat and Jordan). So assume that $i$ and $j$ are even. It suffices to
show that $L_i \cc L_j \sse M_{i+j-1}.$ Take $a \in A,$ $b \in L_{i-1}$ and $c
\in L_j.$ Then $m=[b,a]c$ is a general generator of $L_i \cc L_j.$ By Leibniz
$m=[b,ac]-a[b,c].$ Theorem \ref{th1} applied to the RHS finishes the proof.

\qed

\subsection{Generators of $M$--ideals. An application of maximal containment.}
By Theorem \ref{theorem1} we get the equality
$$
\II(A_2,(i_1,\dots,i_k))=i_1+\dots+i_k-k+1.
$$
Note that this is the first result that explicitly obtains $\II(A_2,i).$ Moreover, in this section we apply this result to find a complete classification of the set of generators of the ideals $M_i(A_2)$ for $i \geq 2.$ Please recall the definition of pure-commutator product from the preliminaries.

First we note two observations: 
\begin{itemize}
\item Pure-commutator Decomposition Property -- every commutator of any length can be presented as a linear combination of pure-commutator products of the PBW degree of the commutator. This follows by the Leibniz rule and induction. For example $[x_1,x_2x_3,x_4]=[x_1,x_2][x_3,x_4]+x_2[x_1,x_3,x_4]+[x_1,x_2,x_4]x_3+[x_2,x_4][x_1,x_3].$
\item Free Permutation Property -- the commutators in every pure-commutator product can be permuted at the cost of an error term of higher degree. This is due to the definition of the bracket and the well-known fact (follows by Jacobi identity and induction) that $[L_i,L_j]\sse L_{i+j}.$ We call it ``free permutation,'' because the error term is zero module most of the ideals in the considerations that follow. For example $[x_1,x_2][x_3,x_4] = [x_3,x_4][x_1,x_2] \mod M_3,$ because $[x_3,x_4][x_1,x_2] = [x_1,x_2][x_3,x_4]-[[x_1,x_2],[x_3,x_4]].$
\end{itemize}
\begin{theorem}[Finite Generation] \label{finitegeneration}
For $i\geq 2$ the ideal $M_i(A_2)$ is finitely generated as a two-sided ideal  by the set $\mc{S}_{i}$ of all pure-commutator products, consisting of commutators of length greater than two, that have a maximal inclusion in $M_i(A_2).$
\end{theorem}

\begin{proof}
It suffices to show that 
$$
\langle \mc{S}_i \rangle := \{ a\ell b \mid a,b \in A_2, \ell \in \mc{S}_i\}=M_i(A_2).
$$
Take an element $m=a\ell b$ from $\langle \mc{S}_i \rangle$ $a,b \in A_2,$ $\ell \in M_{i}(A_2).$ By definition $\ell \in M_i(A_2),$ hence by definition of $M_i(A_2)$ $m \in A_2M_i(A_2)A_2=M_i(A_2).$ Therefore, $\langle \mc{S}_i \rangle \sse M_i(A_2).$

To show the reversed inclusions take an element $m \in \mc{S}_i.$ Through the Pure-commutator Decomposition we present this element as a linear combination of pure-commutator products. Without loss of generality, we consider one such product and use the Free Permutation to present it of the form 
$$
m = \al (x_{j_1}x_{j_2}\cdots x_{j_p})\ell_1 \ell_2 \cdots \ell_q (x_{k_1}x_{k_2}\cdots x_{k_r}),
$$
where $\al \in \mc{K},$ $\ell_s$ is a pure commutator of length $i_s$ for $1 \leq s \leq q,$ and $p,q,r$ are some finite positive integers. We ignore the scalar $\al$ as $\mc{K} \sse A_2.$ Note that the transformed $m$ is a pure-commutator product of $p+q+r$ pure commutators, and its associated tuple is $(1^q,i_1,\dots,i_s,1^r).$ Hence, due to the PBW Argument, $m$ is maximally included in $M_{i_1+\cdots+i_q+q+r-(p+q+r)+1}=M_{i_1+\cdots+i_q-q+1}.$ 

Now we consider cases for the sum $s = i_1+\cdots+i_q-q+1:$
\begin{description}
\item[$s < i$]{Then $m$ is not in $M_i(A_2),$ which is a contradiction with its choice;}
\item[$s = i$]{It means that $m$ is a pure-commutator product that has a maximal inclusion in $M_{\sum}=M_{i},$ i.e. $m \in \langle \mc{S}_i \rangle;$ }
\item[$s > i$]{We argue that we can reduce it to the case $s=i.$ Since an element from $\langle \mc{S}_i \rangle$ is multiplied on both sides by $A_2,$ w.l.o.g. we consider only 
$\ell_1 \ell_2 \cdots \ell_q.$ Due to the same reason, that we can multiply by $A_2$ elements to the sides, it suffices to consider the case $$
\II(A_2,(i_1,\dots,i_q))>i \text{ and }  \II(A_2,(i_2,\dots,i_q))<i.$$ But, since the length of $\ell_1$ is greater than one, we can present it as $\ell_1=[a,\ell],$ where $a \in A_2$ and $\ell$ is another pure-commutator. Then by simple expansion we get 
$$
\ell_1 \ell_2 \cdots \ell_q = a\ell \ell_2 \cdots \ell_q - \ell a \ell_2 \cdots \ell_q,
$$
but through Free Permutation we can ``push'' the element $a$ in the second summand to the farthest left to obtain
$$
a\ell \ell_2 \cdots \ell_q - \ell \ell_2 \cdots \ell_q a.
$$
This sum evaluates under the function $\II$ to $s-1.$ We permuted commutators in exchange of an error term, i.e. we get the combination of the two commutators, which is a pure commutator of the combined degree. However, in this case we reduce $q$ by one, so we can proceed with double induction on $q$ and $s.$ The base of the induction is $q=1$ and $s = \ell_1.$ Then we have the commutator $[a_1,\dots,a_{s}]=[a_1,\dots,a_{s-i},[a_{s-i+1},\dots,a_s]] \in A_2 \cdot \{[a_{s-i+1},\dots,a_s]\} \cdot A_2 \subset \langle \mc{S}_i \rangle,$ obviously. We repeat until we get to $s=i.$
}
\end{description}

\end{proof}

By convention we set $S_1:=\{1\}.$ Obviously this generates $M_1\equiv A$ as a two-sided ideal.
\subsection{An application to Universal Lie Nilpotent Associative Algebras}
Universal Lie Nilpotent Associative algebras are first introduced by Etingof, Kim and Ma \cite{EKM}, and are defined by $Q_{i}(A_n):=A_n/M_i(A_n).$ The algebra $Q_i$ is the most free algebra that satisfies $[a_1,\dots,a_i]=0.$ The paper motivates the need to study the $Q$-algebras, because they help us better understand Lie nilpotent algebras $B$ of any index $i$ (i.e. $M_{i+1}(B)=0$). In particular, Etingof, Kim and Ma give a description of $Q_4(A_n)$ in terms of generators and relations. We follow the ideas in the paper to obtain a description of $Q_i(A_2)$ for any $i\geq 2.$

\begin{theorem}
For any $i\geq 2$ the ideal $M_i(A_2)$ is generated by the Lie relations $P = 0$ for any $P \in \mc{S}_i.$ Moreover, $Q_i(A_2)$ and $A_2/\langle  \mc{S}_i \rangle$ are isomorphic (as vector spaces).
\end{theorem}

\begin{proof}
Let $C_i^{(2)}:=A_2/\langle  \mc{S}_i \rangle.$ All of the relations $P=0$ follow from Theorem \ref{theorem1}. Hence, there is a natural surjective homomorphism $\theta: C_i^{(2)} \twoheadrightarrow Q_i(A_2).$ We are left to show that $\theta$ is an isomorphism. It suffices to prove that $[a_1,\dots, a_{i-1}]$ is a central element in $C_i^{(2)}.$ We have that $[a_1,\dots, a_{i-1}] = 0$ in $Q_{i-1}(A_2).$ Therefore, $[a_1,\dots, a_{i-1}] \in M_{i-1}(A_2).$ Hence, we need to show that the generators of $M_{i-1}(A_2)$ are central. We know the generating set by Theorem \ref{finitegeneration}. Thus, take $P' \in \mc{S}_{i-1}.$ W.l.o.g. to prove centrality we consider only $[P',x_1].$ Due to the Pure-Commutator Decomposition Property we expand $[P',x_1]$ as a linear combination of pure-commutator products $P,$ where $P$ is obtained from $P'$ by extending one of the commutators by one slot, i.e. $P$ evaluates to $i$ under $\II$. Hence, $[P',x_1]=0$ in $C_i^{(2)},$ which completes the proof.
\end{proof}

\section{Discussion about maximal containment for $n=3.$} \label{discussion}
This is a good place to recap the main results of the paper so far. On the one hand, we gave the PBW argument that presents us with an upper bound on $\II(A_n,i).$ On the other hand, for the special case of $n=2$ we used the Pigeonhole identity to set a lower bound on $\II(A_2,i)$ and saw that this is sharp. This lead us to the maximal inclusions of lower central series ideals for $n=2.$ 

It is tempting to find a sharp lower bound for $n=3$ too. Indeed, the Pigeonhole identity gives us $M_2(A_3) \cdot M_2(A_3) \sse M_3(A_3).$ Moreover, by Theorem \ref{theorem2} we have that $M_i \cdot M_j \sse M_{i+j-1}$ for $i$ or $j$ odd. Hence, it suffices to consider $i$ and $j$ even. The first open question is whether $M_2(A_3) \cdot M_4(A_3)$ lives in $M_5(A_3).$
\begin{itemize}
\item $[x,y][z,[x,[x,y]]] \in M_5$ (this is due to Jacobi identity applied on the second commutator and the Pigeonhole identity); 
\item $[x,y][z,y,x,y] \in M_5:$ The proof is as follows. Expanding the second commutator via Jacobi identity we get (up to sign) $[[x,y],[z,y]]$ and $[y,[[x,y],z].$ The latter commutator of the expansion allows us to use the Pigeonhole identity. The former is a bit trickier. 
Consider $[x,y][[x,y],[z,y]].$ Note that, due to the Leibniz rule, $a[a,b]=[a,ab]$ for any $a$ and $b.$ Then $[x,y][[x,y],[z,y]] = [[x,y],[x,y][z,y]] \in [L_2,M_3] \subset M_5.$ We used the Pigeonhole identity and Bapat and Jordan's observation. \qed
\item One can easily give an argument (reduction of degree via the Leibniz rule) that to prove $M_i\cdot M_j \sse M_{k}$ for some $i,j,k$ and any algebra admitting the Leibniz rule, it is sufficient to do it for pure commutators only.
\end{itemize}

The stumble is the following commutator $[x,y][z,z,z,x].$ Is it in $M_5?$ We tried to use the trick from the bullet points but we end up with the same thing. The problem is that the common generator, $x,$ is ``farther down to the left.'' The ultimate goal is to bring it as close as possible to the other $x$ so that we can apply the Pigeonhole identity. We believe this is not a trivial task, and maybe more identities would be needed. 
 
However, computer calculations for small degree tell us that $[x,y][z,z,z,x]$ is in $M_5.$ More generally we conjecture the following: 

\begin{conjecture}
Given the $k$-tuple
$i=(i_1,\dots,i_k),$ such that $i_p$ is greater than or equal to two for
$p=1,\dots ,k$ the following holds
$$
\II(A_3,i)=i_1+\dots+i_k-k+1.
$$
\end{conjecture}

The significant connection to the case $n=2$ is that we have a small amount of generators, which allows us to have a common generator between two commutators.

\section{An approach to non-containment} \label{approach}
In order to improve the classification of lower central series containments, we
need to find $\II(A_n,(i,j))$ when $i$ and $j$ are both even. We know that in
this case $M_i \cdot M_j \subset M_{i+j-2}.$ 
Etingof, Kim, and Ma ~\cite{EKM} conjectured that---generally---
$M_i \cdot M_j \subset M_{i+j-1}$ if and only if $i$ or $j$ is odd (see Conjecture 3.6 in the arXiv paper). In 2010 Bapat and Jordan \cite{BJ} prove 
the ``if'' part of the statement (see Corollary 1.4 in the arXiv paper). Hence, the conjecture boils down to showing that $\II(A_n,(i,j))=i+j-2$ when 
$i$ and $j$ are both even. This is an interesting glitch of the general pattern. An approach might be to seek an associative algebra which does not admit the desired inclusion. If the algebra is chosen \emph{properly}, by a standard mapping to the free associative algebra the conjecture would hold. In what follows we demonstrate our approach to the problem via the Lie algebra $\mathfrak{sl}(2)$ with its usual basis
$$
e=\begin{pmatrix}
0 & 1 \\
0 & 0
\end{pmatrix},
f=\begin{pmatrix}
0  & 0\\
-1 & 0
\end{pmatrix},
h=\begin{pmatrix}
1 & 0 \\
0 & -1
\end{pmatrix},
$$
and relations
$$
[h,e]=2e, [h,f]=-2f, [e,f]=h.
$$

\begin{theorem} \label{thmzz}
The following holds:
        $$M_i(A_n) \cdot M_j(A_n)$$
is not contained in $M_{i+j}(A_n)$ when $i$ and $j$ have the same parity for
$n\geq 2.$ Furthermore, the statement holds
when $i$ and $j$ have different parity, but $n\geq 3.$
\end{theorem}
\begin{proof}
Consider the homomorphism
$$
\varphi: A_2 \longrightarrow \mathrm{Mat}_2(\mathcal{K})
$$
defined by $\varphi(x_1)=e+f$ and $\varphi(x_2)=h.$ We extend $\varphi$ to a
homomorphism
$$
\varphi: A_n \longrightarrow \mathrm{Mat}_2(\mathcal{K}).
$$
Since the lowest degree elements lie in the $L$ elements it suffice to show that
the product $L_i(A_n) \cdot L_j(A_n)$ is not contained in $L_{i+j}(A_n).$ Take
$$
\l=\mathrm{ad}^{i-1}(x_2)(x_1)\cdot \mathrm{ad}^{j-1}(x_2)(x_1) \in L_i \cdot
L_j.
$$
Moreover, $\varphi(\l)=\mathrm{ad}^{i-1}(h)(e+f)\cdot
\mathrm{ad}^{j-1}(h)(e+f).$
Using the relations of $\mathfrak{sl}(2)$ and induction we obtain the following
identities:
\[ \mathrm{ad}^{k}(h)(e+f) = \left\{ \begin{array}{ll}
         2^{k}(e-f) & \mbox{for odd $k$};\\
         2^{k}(e+f) & \mbox{for even $k$}.\end{array} \right. \]
Hence when $i$ and $j$ are both odd $\l$ maps to $2^{i+j}(e-f)^2$, and thus
$\mathrm{trace}(\varphi(l))=2^{i+j+1}.$ Furthermore, when $i$ and $j$ are both
even, $\l$ maps to $2^{i+j}(e+f)^2,$ which also has nonzero
$\mathrm{trace}(\varphi(\l))=-2^{i+j+1}.$
However, since commuting implies zero trace, all elements of $L_{i+j}(A_n)$ are
mapped to matrices with trace equal to zero. This argument proves the statement
for $i$ and $j$ of the same parity.

Moreover, if $i$ and $j$ are of different parity and $n\geq 3$, we can take the
same homomorphism $\varphi,$ but we map the generators differently:
$$
\varphi(x_1)=e, \varphi(x_2)=h, \varphi(x_3)=f.
$$
Now, the element $\mathrm{ad}^{i-1}(x_2)(x_1)\mathrm{ad}^{j-1}(x_2)(x_3)$ maps
to a matrix that does not have trace zero.
\end{proof}

\begin{remark}
The PBW argument (Proposition \ref{proposition1}) is a stronger result, but the method in Theorem \ref{thmzz} is demonstrative. Furthermore, Deryabina and Krasilnikov ~\cite{kdrumen} recently proved the ``only if'' part of the conjecture by showing---more generally---that $M_i(\mc{K}'\langle X \rangle) \cdot M_j(\mc{K}' \langle X \rangle) \nsubseteq M_{i+j-1} ( \mc{K}' \langle X \rangle )$ if both $i$ and $j$ are even, where $\mc{K}'\langle X \rangle$ is the free unital associative algebra over $\mc{K}'$ freely generated by the (infinite and countable) set of generators $X=\{x_1,x_2,\dots\}.$ Note that $\mc{K}'$ can be \emph{any} field.
In the special case $\mc{K}':=\mc{K}$ (of characteristic zero), taking a finite set of generators, they show that to prove the conjecture it suffices to use the algebra $$\mc{A}:= E \otimes E_{i+j-4},$$
where $E$ is the infinite-dimensional unital Grassmann algebra and $E_{i+j-4}$ is the unital Grassmann sub-algebra generated by the first $i+j-4$ generators. In fact the authors prove that $M_{i+j-1}(\mc{A})$ is zero, while $M_i(\mc{A}) \cdot M_j(\mc{A})$ is not. This statement fails for small numbers $n$ of variables. In particular it requires $n \geq i + j.$

\end{remark}

We conclude with a conjecture for the general case of tuples, consisting only of two's, i.e. $(2,\dots,2)\in \bb{N}^k.$ We believe this is true due to combinatorial considerations. We checked the validity of the conjecture for small cases.

\begin{conjecture}
Given the $k$-tuple $j:=(2,\dots,2)\in \bb{N}^k$ consisting only of two's  
\[ \II(A_n,j) = \left\{ \begin{array}{ll}
         \max\left \{ 2\left \lceil \frac{k}{2}-\frac{n}{4} \right \rceil, 0 \right \} + 2 & n \geq 4 \text{ and } n \text{ even; }\\
         \max \left \{ 2\left \lceil \frac{k}{2}-\frac{n}{4}+\frac{1}{4} \right \rceil, 0 \right \} + 2 & n \geq 5 \text{ and } n \text{ odd; }\\
         k + 1 & \text{otherwise.}\end{array} \right. \]
\end{conjecture}

Note that setting $i_1=\cdots = i_k=2$ in Theorem \ref{theorem1} we get that $\II(A_2,(i_1,\dots,i_k))=\II(A_2,(2,\dots,2))=2k-k+1=k+1,$ which is compatible with the conjecture.

\section{General bounds on $\II(A_n,i).$ Summary} \label{bounds}

We try to obtain some easy containments for an arbitrary long product of $M$--ideals, i.e. we consider $\II(A_n,(i_1,\dots,i_k))$ for $k\geq 2.$ Suppose that exactly $p$ of the elements $i_1, \dots, i_k$ are odd ($p\leq k$). Hence, we have $k-p$ even indices. We have Theorem \ref{theorem3} and Theorem \ref{theorem2} in our arsenal. These theorems take two ideals and include them in a bigger one. Therefore, we have $k$ steps until we obtain the complete inclusion. On each step we take two indices and produce a new one. If we take two even indices, then we use Theorem \ref{theorem3} and get a new even index. This way after a step we have one even index less. If we have two odd indices, we use Theorem \ref{theorem2} and obtain an odd index, i.e. after the step we have one odd index less. If we have an odd index and and even index, we get an even index, i.e. after the step we have one odd index less. Therefore, when we use the stronger inclusion we always reduce the number of the odd indices by one. We start with $p$ odd indices. It follows that 
$$
\II(A_n,(i_1,\dots ,i_k)) \geq i_1+\cdots i_k - 2(k-1)+p= \sum_{v=1}^{k} i_v-2k+p+1.
$$
Note that the procedure is invariant of the order of the combinations of indices, because of the above-mentioned combinatorial considerations. This fact also follows algebraically by the Free Permutation Property. Moreover, if $p=k$ then due to the PBW argument we get a maximal inclusion.

In conclusion, we collect the bounds (and equalities) on $\II$ that follow from previous results, combined with this paper: 

\begin{itemize}
\item $\II(A_n,(m,l))\geq m+l-2;$  
\item $\II(A_n,(m,l))\geq m+l-1$ when $m$ or $l$ is odd; 
\item $\II(A_n,(i_1,\dots,i_k))\leq i_1+\cdots+i_k-k+1;$
\item $\II(A_2,(i_1,\dots,i_k))=i_1+\cdots+i_k-k+1;$
\item $\II(A_n,(i_1,\dots,i_k))\geq i_1+\cdots + i_k - 2k + p + 1$ \text{ if } $p$ \text{ indices are odd;}
\item $\II(A_n,(i_1,\dots,i_k))= i_1+\cdots + i_k - k+1$ \text{ if all indices are odd;}
\item $\II(A_3,(2,2))=3.$
\end{itemize}

\section{Applications to $PI$--algebras} \label{applications}

It turns out that results on lower central series containments are very useful in
the study of the following specific quotients
$$
\{B_i(A)=L_i(A)/L_{i+1}(A)\}_{i\geq1} \ \text{and} \
\{N_i(A)=M_i(A)/M_{i+1}(A)\}_{i\geq1}.
$$
Finding the structure of these vector spaces is a hard open problem that
involves the representation theory of $W_n,$ the polynomial vector fields on
$\mathbb{C}^n,$ and, in similar language, the representation theory of the
general linear group on $n$ generators $GL_n.$ There are a few results that shed
light on the structures of the $B$--series and the $N$--series (\cite{FS},
\cite{DKM}, \cite{EKM}, \cite{DE}). 

To apply the results from this paper we replace $A$ with certain PI-algebras
$$
R_{i,j}(A_n):=A_n/(M_i(A_n) \cdot M_j(A_n)) \text{ (we may omit $A_n$ when $n$ is implied). }
$$ 

We similarly define the lower central series for $R_{i,j},$ i.e. $L_1(R_{i,j})=R_{i,j}$ and
$L_i(R_{i,j}):=[R_{i,j},L_{i-1}(R_{i,j})].$ Hence, we obtain the analogous $M,$ $B$ and $N$ ideal series. The PI-Algebras $A/(M_i(A)\cdot
M_j(A))$ are of particular interest in this setting because the structure of the commutator ideals $M_i$ allows the action of $W_n$ to descend to these quotients. In particular, we hope to understand the lower central series quotients in terms of the representation theory of $W_n$.

We consider $R_{i,j}(A_n)$ as $A_n+\mc{I}_n,$ where $\mc{I}_n:=M_i(A_n) \cdot M_j(A_n)$ is the ideal that generates the equivalence classes. Then easily, just by playing with inclusions, one can show that Theorems
\ref{theorem1} and \ref{theorem2} reveal that generally $B_i \cong N_i$ for many
PI-algebras. 

\begin{theorem}[Isomorphism Property]
The following holds
$$
B_i(R_{j,2}) \cong N_i(R_{j,2}) \text{ as graded vector spaces }
$$
for all positive integers $i \geq t$, where $t=2 \left \lceil \frac{j+1}{2} \right \rceil$.
\end{theorem}

Moreover, $PI$-algebras allow us to obtain explicit description of the vector spaces $N,$ which give us an explicit approximation of the corresponding $PI$-algebra. We use the identity
$$
[a,b,\ell]=[b,a,\ell]+[[a,b],\ell],
$$
which combined with the Jacobi identity and the anti-commutative law allows us to permute slots of the commutator with some error terms ($[[a,b],\ell]$) as a trade-off. 

For example, we have such a description of the free metabelian associative algebra $R_{2,2}(A_n)$ (comprehensively studied by Drensky in \cite{D}). We use the metabelian identity $[a,b][c,d]=0$ to get that $[a,b,\ell]=[b,a,\ell] \mod M_2 \cdot M_2.$ This allows us to permute freely the elements in the commutator, and thus find a span for $M_i(R_{2,2}).$ Induction leads us to the following result.

\begin{theorem}[The Structure Theorem]
        For a fixed $r\geq 1$ a basis for the elements $N_r(R_{2,2})$ is
            $$
                \{x_1^{a_1}\cdots x_n^{a_n}[x_{i_1},\dots, x_{i_r}]\},
            $$
        where $a_j\geq 0$ for all $j$, $i_1>i_2$ and $i_2 \leq \cdots \leq i_r$ in the
        commutators of length $r$.
    \end{theorem}

Note that this is an explicit description of the structure for the $N$--series of the free metabelian associative algebra, and it is for any number of generators $n$ and any $r.$ We extend this idea using maximal inclusions for $n=2$ to prove other results, based on the language of $GL_2$-representations.

\begin{theorem}[The Structure Theorem for $R_{2,3}$]
For $r$ greater than four we have that for the algebra $R_{2,3}(A_2)$ the following holds:
$$N_r(R_{2,3}(A_2)) \cong \left(\sum_{i=0}^{\infty} Y_2(i)\right) \otimes \left( Y_2(r-1,1) \oplus Y_2(r-3,1) \otimes Y_2(1,1) \right).$$
\end{theorem}

In this theorem $Y_2(r-1,1)$ is the $GL_2$ module associated with the partition $(r-1,1).$ We achieve this result through the maximal inclusion of the $M$--ideals and the following bijection map 
between ``sorted'' commutators, i.e. $i_1> i_2 \leq i_3 \leq \cdots \leq i_r,$ and semi-standard Young tableaux: 

$$
[x_{i_1},\dots,x_{i_r}] \mapsto \young(\I \blah \III \cdots \IIIIII,\IIIIIII ).
$$

\begin{remark}
Loosely speaking, in the above-mentioned problems we have two \emph{parameters}: the number of generators, and the length of the commutators. It seems that if we fix one of the parameters to be sufficiently small we obtain results that provide us with a good classification of structure and behavior. We do not know how to think about explicit description when we set the two parameters free.
 
Finally, just for fun, we leave the discussion open with these two interesting
and potentially strong identities.
$$
[x,[y,[y,x]]]=[y,[x,[y,x]]]
$$
$$
[xy,[y,x]]=[yx,[y,x]]
$$
\end{remark}

\section*{Acknowledgements}
My deepest gratitude goes to Nathan Harman, graduate student at the Massachusetts
Institute of Technology, for the fruitful discussions (especially about the PBW
theorem) and his comments on this paper. Furthermore, I sincerely acknowledge
Professor Pavel Etingof for suggesting the consideration of the PBW theorem and
providing constant encouragement.

\bibliographystyle{plain}
\bibliography{biblio}

\end{document}